\documentclass{amsart}
\usepackage{amsfonts,amssymb,amsmath,amsthm}
\usepackage{url}
\usepackage{enumerate}
\usepackage{xcolor}

\usepackage{amsbsy,hyperref,cite,epstopdf}
\usepackage[all,v2,cmtip]{xy}

\newtheorem{theorem}{Theorem}[section]
\newtheorem{lemma}[theorem]{Lemma}

\theoremstyle{definition}
\newtheorem{definition}[theorem]{Definition}
\newtheorem{remark}[theorem]{Remark}
\newtheorem{conjecture}[theorem]{Conjecture}

\newtheorem{example}[theorem]{Example}
\newtheorem{algorithm}[theorem]{Algorithm}
\numberwithin{equation}{section}

\author{Injo Hur and Yeansu Kim}
\address{
Department of Mathematics Education\\
Chonnam National University\\
Gwangju city, Korea}
\email{injohur@jnu.ac.kr}
\address{
Department of Mathematics Education\\
Chonnam National University\\
Gwangju city, Korea}
\email{ykim@jnu.ac.kr}

\begin{document}

\title[Balanced non-transitive dice]{Possible Probability and irreducibility of balanced non-transitive dice}
\date{April 2020}

\begin{abstract}
We construct irreducible balanced non-transitive sets of $n$-sided dice for any positive integer $n$, which was raised in \cite[Question 5.2]{SS17}. One main tool of the construction is to study so-called fair sets of dice. Furthermore, we also study the distribution of the probabilities of balanced non-transitive sets of dice. For a lower bound, we show that the probability could be arbitrarily close to $\frac{1}{2}$ and for a upper bound, we construct a balanced non-transitive set of dice whose probability is $\frac{1}{2} + \frac{13-\sqrt{153}}{24} \approx \frac{1}{2} + \frac{1}{9.12}.$
\end{abstract}
\maketitle

\section{Introduction}
A non-transitive set of dice, first introduced by Gardner \cite{G70} and further studied in \cite{S94}, is triples of labeled dice, $A$, $B$, and $C$, with the property that all three probabilities that $A$ rolls higher than $B$, that $B$ rolls higher than $C$, and that $C$ rolls higher than $A$ are greater than $\frac{1}{2}$. We write this as $P(A>B)>\frac{1}{2}, P(B>C)>\frac{1}{2}$, and $P(C>A) > \frac{1}{2}$. It is called non-transitive because if we define the relation $X \sim Y$ as $P(X>Y) > \frac{1}{2}$ for $X, Y \in \{A, B, C\}$, then the relation $\sim$ is not transitive. For example, the following set of $6$-sided dice is non-transitive since $P(A>B)=P(B>C)=P(C>A)=\frac{19}{36}$:
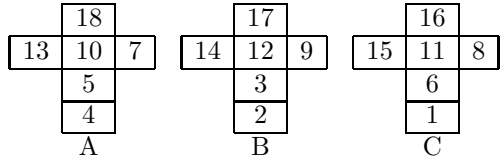
\[
\begin{tabular}{ccc}
\cline{2-2}
\multicolumn{1}{c|}{}    & \multicolumn{1}{c|}{18} &                        \\ \hline
\multicolumn{1}{|c|}{13} & \multicolumn{1}{c|}{10} & \multicolumn{1}{c|}{7} \\ \hline
\multicolumn{1}{c|}{}    & \multicolumn{1}{c|}{5}  &                        \\ \cline{2-2}
\multicolumn{1}{c|}{}    & \multicolumn{1}{c|}{4}  &                        \\ \cline{2-2}
                         & A                       &                       
\end{tabular}
\ \ \ 
\begin{tabular}{ccc}
\cline{2-2}
\multicolumn{1}{c|}{}    & \multicolumn{1}{c|}{17} &                        \\ \hline
\multicolumn{1}{|c|}{14} & \multicolumn{1}{c|}{12} & \multicolumn{1}{c|}{9} \\ \hline
\multicolumn{1}{c|}{}    & \multicolumn{1}{c|}{3}  &                        \\ \cline{2-2}
\multicolumn{1}{c|}{}    & \multicolumn{1}{c|}{2}  &                        \\ \cline{2-2}
                         & B                       &                       
\end{tabular}
\ \ \ 
\begin{tabular}{ccc}
\cline{2-2}
\multicolumn{1}{c|}{}    & \multicolumn{1}{c|}{16} &                        \\ \hline
\multicolumn{1}{|c|}{15} & \multicolumn{1}{c|}{11} & \multicolumn{1}{c|}{8} \\ \hline
\multicolumn{1}{c|}{}    & \multicolumn{1}{c|}{6}  &                        \\ \cline{2-2}
\multicolumn{1}{c|}{}    & \multicolumn{1}{c|}{1}  &                        \\ \cline{2-2}
                         & C                       &                       
\end{tabular}
\]

In \cite{SS17}, Schaefer and Schweig constructed balanced non-transitive sets of $n$-sided dice for any positive integer $n \geq 3$. Main idea of the construction in \cite{SS17} is to combine several balanced non-transitive sets of dice. Therefore, the sets of dice that are constructed in \cite{SS17} are reducible. (See Definition \ref{def:irreducible} for the meaning of being irreducible.) and authors in \cite{SS17} questions whether there exists an irreducible balanced non-transitive sets of $n$-sided dice for all $n$.
One main purpose of the paper is to construct irreducible non-transitive sets of $n$-sided dice for any positive integer $n$. Our main idea of the construction is to use so-called fair sets of $2$-sided dice. Here, being fair means that probabilities $P(A>B), P(B>C)$, and $P(C>A)$ are all $\frac{1}{2}$. Although we used $2$-sided dice to construct new sets of dice, understanding fair sets of $n$-sided dice for any positive integer $n$ seems to be an important step to understand all irreducible balanced non-transitive dice. We also study fair sets of $n$-sided dice for any positive integer $n$. 

The second purpose of the paper is to study possible probabilities of balanced non-transitive sets of $n$-sided dice, i.e., possible value of $P(A>B)$. As far as we know, previous known constructions of balanced non-transitive sets of dice have probability $\frac{1}{2} < P(A>B) < \frac{3}{5}=\frac{1}{2} + \frac{1}{10}$. We first show that there exists a balanced non-transitive set of dice such that $P(A>B)-\frac{1}{2}>0$ becomes arbitrary small, i.e., $\frac{1}{2}$ is a sharp lower bound for balanced non-transitive set of dice. We expect that $\frac{1}{2}+\frac{1}{10}$ is not a sharp upper bound and make a conjecture that $\frac{1}{2} < P(A>B) < \frac{1}{2} + \frac{1}{9}$. To support the conjecture, we first explicitly calculate and prove that both the probability $P(A>B)$ in our paper and the one in \cite{SS17} are less than $\frac{1}{2}+\frac{1}{9}$. We also provide another new construction of a balanced non-transitive set of dice such that $P(A>B) \approx \frac{1}{2} + \frac{1}{9.12}$, which is less than $\frac{1}{2}+\frac{1}{9}$ and is, as far as we know, the maximum among known constructions of balanced non-transitive sets of dice.

Let us now describe the content of our paper. After we introduce notation and preliminaries in Section \ref{notation}, we study fair sets of dice in Section \ref{fair sets}. In Section \ref{irreducible}, we provide construction of irreducible balanced sets of $n$-sided dice for any positive integer $n$ using fair sets of $2$-sided dice. In Section \ref{probability}, we calculate and study possible probability $P(A>B)$ for balanced non-transitive sets of dice.

\section{Notation and preliminaries}\label{notation}

We briefly recall definitions and notations in \cite{SS17} that we are going to use.

\begin{definition}
Fix an integer $n>0$. A set of $n$-sided dice is a collection of three pairwise-disjoint sets $A, B,$ and $C$ with $|A|=|B|=|C|=n$ and $A \cup B \cup C = \{1, \cdots, 3n \}$. We think of die $A, B$, and $C$ as being labeled with the elements of $A, B$, and $C$, respectively and we assume that each dice is fair (i.e. the probability of rolling any one of its numbers is $\frac{1}{n}$.)
\end{definition}

\begin{definition}\label{def} A set of dice is called
\begin{itemize} 
    \item  {\it balanced} if $P(A>B)=P(B>C)=P(C>A)$. 
    \item  {\it non-transitive} if each of $P(A>B), P(B>C)$, and $P(C>A)$ exceeds $\frac{1}{2}$.
    \item  {\it fair} if $P(A>B)=P(B>C)=P(C>A)=\frac{1}{2}$
\end{itemize}
\end{definition}

\begin{definition}\label{word}
If $D=(A, B, C)$ is a set of $n$-sided dice, define a word $\sigma(D)$ by the following rule: the i$^{th}$ letter of $\sigma(D)$ corresponds to the die on which the number $i$ labels a side.
\end{definition}

For a word $\sigma$, a number of alphabets in the word, which we call the length of $\sigma$, is denoted by $|\sigma|$. A number of $A$ (resp. $B$, $C$) in $\sigma$ is denoted by $|A|_{\sigma}$ (resp. $|B|_{\sigma}$, $|C|_{\sigma}$). For example, if $\sigma$ corresponds to a set of $n$-sided dice, $|A|_{\sigma}=|B|_{\sigma}=|C|_{\sigma} = n$. 

Let $D=(A, B, C)$ be a set of $n$-sided dice and $\sigma$ be its corresponding word (Definition \ref{word}). Then we denote by $P_D(A>B)$ or $P_{\sigma}(A>B)$ the probability that the number rolled on A is greater than the number rolled on B when we roll $A$ and $B$. Similarly, we define $P_D(B>C), P_{\sigma}(B>C), P_D(C>A)$, and $P_{\sigma}(C>A)$.

\begin{definition}\label{defs w.r.t. N}
For a set of $n$-sided dice $D$ (and its corresponding word $\sigma$ with length $3n$), we let $N_D(A>B):=n^2P_D(A>B), N_D(B>C):=n^2P_D(B>C)$, and $N_D(C>A):=n^2P_D(C>A)$. Similarly we also define $N_{\sigma}(A>B)$, $N_{\sigma}(B>C)$, and $N_{\sigma}(C>A)$.
\end{definition}

It is by definition that 
$D=(A,B,C)$ is balanced if $N_D(A>B)=N_D(B>C)=N_D(C>A)$. And $D=(A,B,C)$ is non-transitive if $N_D(A>B), N_D(B>C), N_D(C>A) > \frac{n^2}{2}.$

\begin{remark}
$N_D(A>B)$ in Definition \ref{defs w.r.t. N} represents the number of consequences that a die $A$ wins a die $B$ when we consider all possible outcomes when we roll $A$ and $B$. Note that $N_D(A>B)$ is the same as the notation $\displaystyle\sum\limits_{s_i=A} q_{\sigma(D)}^+(s_i)$ in \cite{SS17}.
\end{remark}

\begin{example}
Let $D$ be the following set of $3$-sided dice:
\begin{eqnarray*}
A &=& 9 \ 5 \ 1 \\
B &=& 8 \ 4 \ 3 \\
C &=& 7 \ 6 \ 2.
\end{eqnarray*}
Then $\sigma(D)=ACBBACCBA$ and $P(A>B)=P(B>C)=P(C>A)= \frac{5}{9}$. Therefore this set of dice is balanced and non-transitive.
\end{example}

\begin{definition}
The concatenation of two words of $\sigma$ and $\tau$ is simply the word $\sigma$ followed by $\tau$, denoted by $\sigma \tau$.
\end{definition}
We recall a recursive relation of $N(A>B), N(B>C)$, and $N(C>A)$. The following is in the proof of Lemma 2.4 in \cite{SS17}.
\begin{lemma}[(1) in \cite{SS17}]
Let $\sigma$ and $\tau$ be two words that correspond to two sets of dice, respectively. Let $|\sigma|=3m$ and $|\tau|=3n$ (i.e., corresponding dice are $m$-sided and $n$-sided, respectively). Then we have
\begin{equation}\label{recursion formula}
    N_{\sigma\tau}(A>B) = N_{\sigma}(A>B) + N_{\tau}(A>B) + mn.
\end{equation}
\end{lemma}

\section{Fair sets of dice}\label{fair sets}

In the following two lemmas ``probabilities" means $P(A>B)$, $P(B>C)$ and $P(C>A)$ and $x$, $y$ or $z$ is one of $A$, $B$ or $C$.
\begin{lemma}\label{symm exchange}
Assume that $\sigma(D)$ has $xy$ and $yx$, i.e., $\sigma(D)= \cdots xy \cdots yx \cdots$. Then exchanging the orders of $xy$ and $yx$ at the same time, i.e., $\sigma(\tilde{D})= \cdots yx \cdots xy \cdots$ does not change the ``probabilities".
\end{lemma}
\begin{lemma}\label{xyz exchange}
Assume that $\sigma(D)$ is a word having three different letters, say $x$, $y$ and $z$. Then $xyz \, \sigma(D)$ and $\sigma(D)\, xyz$ have the same ``probabilities" if and only if $D$ has the same numbers of $x$, $y$ and $z$. 
\end{lemma}


\begin{definition}\label{similar}
Two words are called {\it similar}, $\sigma(D_1) \sim \sigma(D_2)$, if  $\sigma(D_1)$ can be re-written to $\sigma(D_2)$ by exchanging the orders in Lemmas \ref{symm exchange} and \ref{xyz exchange}.
\end{definition}
Note that if $\sigma(D_1) \sim \sigma(D_2)$, then two words have the same length and  $P_{D_1}(A>B)=P_{D_2}(A>B), P_{D_1}(B>C)=P_{D_2}(B>C)$, and $P_{D_1}(C>A)=P_{D_2}(C>A)$.

\begin{example}
Lemma \ref{symm exchange} and Lemma \ref{xyz exchange} indicate that $\cdots xy \cdots yx \cdots \sim \cdots yx \cdots xy \cdots$ and $xyz \, \sigma(D) \sim \sigma(D)\, xyz$ if and only if  $|x|_{\sigma(D)}=|y|_{\sigma(D)}=|z|_{\sigma(D)}$.
\end{example}

\begin{conjecture}[Fair dice]\label{fair dice} 
If $D$ is fair, then the length of $D$ is a multiple of 6 and it is similar to a product of $xyzzyx$. 
\end{conjecture}
Note that $\tau:=ABCCBA$ is a fair set of $2$-sided dice, which will be used in Section \ref{irreducible}. 

\begin{example}
Let $\sigma(D)$ be $AABBCCCCBBAA$. It is fair and 
\[
AABBCCCCBBAA \sim ABABCCCCBABA \sim ABACBCCBCABA
\]
\[\sim ABCABCCBACBA \sim ABCCBAABCCBA =(ABCCBA)^2.
\]
\end{example}

To support Conjecture \ref{fair dice}, let us consider only {\it two dice} case, i.e., a fair word with only two letters $A$ and $B$. In this case a set of dice and a corresponding word are called fair if $P(A>B)=P(B>A)=\frac{1}{2}.$
\begin{theorem}
Assume that a given fair word $\sigma_f$ has $4m$ letters with $|A|=|B|=2m$. Then $\sigma_f\sim (ABBA)^m$.
\end{theorem} 
\begin{proof}
We may  assume that $\sigma_f$ starts with $A$ (if not, we exchange the notations of $A$ and $B$). Then let us claim that the factor $ABBA$ can be extracted to the front without changing probabilities, i.e., $\sigma_f\sim ABBA\, \omega_f$ where $\omega_f$ is a fair word having $4(m-1)$ letters with $|A|=|B|=2(m-1)$. Mathematical induction with this fact leads us to the result that $\sigma_f\sim (ABBA)^m$. 

Observe first that $\sigma_f\sim AB\,\omega_0$ where $\omega_0$ is a word such that $AB\,\omega_0$ is fair. If $\sigma_f=AB\,\omega_0$, then we are done. Suppose not, i.e., $\sigma_f=A\cdots A B \,\omega_1$. We would like to see that $\omega_1$ has a sequel $BA$. If $\omega_1$ would not have $BA$, then the given fair dice should be expressed by $\sigma_f=A\cdots A B A\cdots A B\cdots B$, which is not fair at all. Due to having $BA$ in $\omega_1$, exchanging $AB$ and $BA$ together will hold probabilities, i.e., $\sigma_f\sim A\cdots BA\,\omega_2$, where $\omega_2$ is the word obtained from $\omega_1$ by replacing $BA$ by $AB$. A similar argument works to reveal that $\omega_2$ has $BA$ and therefore $\sigma_f\sim A\cdots BAA\,\omega_3$. Keep doing this procedure until $\sigma_f\sim AB\,\omega_0$.

Next see that $\sigma_f\sim ABBA \,\omega_f$ as claimed. If $\omega_0$ starts with $BA$, then we are done. If not, there are three cases, that is, $\omega_0$ is one of $AB\,\omega_4$, $A\cdots AB\,\omega_5$, $B\cdots BA\,\omega_6$. Since a similar argument can be applied, let us consider the first case when $\omega_0=AB\,\omega_4$. If $\omega_4$ would not have $BA$, then it should be expressed by $A\cdots A B\cdots B$. This means that $\sigma_f\sim ABABA\cdots A B\cdots B$, which is impossible to be fair. So $\omega_4$ has a sequel $BA$. Exchanging $BA$ (in $\omega_4$) to $AB$ (in front of $\omega_4$) shows that $\sigma_f\sim ABBA\,\omega_f$. As mentioned, in the other two cases, a similar argument reveals that $\sigma_f\sim ABBA\,\omega_f$ as claimed. 
\end{proof}

\section{Irreducible set of dice}\label{irreducible}
In this section, we construct irreducible balanced non-transitive sets of $n$-sided dice for any positive integer $n$. Note that this answers Question 5.2 in \cite{SS17}. 

\begin{definition}\label{def:irreducible}
A balanced non-transitive word (and a corresponding set of dice) is called irreducible if there do not exist balanced non-transitive words $\sigma_1$ and $\sigma_2$ (both nonempty) such that $\sigma = \sigma_1\sigma_2$.
\end{definition}
\begin{lemma}\label{baby construction}
Let $\sigma$ be a balanced non-transitive word and $\tau=ABCCBA$. Then $\tau\sigma$ is balanced and non-transitive.
\end{lemma}
\begin{proof}
$\tau$ is balanced but not non-transitive since $P(A>B)=P(B>C)=P(C>A)= \frac{1}{2}$. Lemma 2.4 in \cite{SS17} implies that $\tau\sigma$ is balanced. It remains to prove that it is non-transitive. Let $|\sigma|=3n$. For $x=A, B,$ or $C$, the equation \eqref{recursion formula} implies
\[
N_{\tau\sigma}(A>B) = 2 + N_{\sigma}(A>B) + 2n > 2 + \frac{n^2}{2} + 2n = \frac{(n+2)^2}{2}.
\]
Therefore $\tau\sigma$ is also non-transitive (Definition \ref{defs w.r.t. N}).
\end{proof}

\begin{lemma}\label{main construction}
Let $\sigma$ be a balanced non-transitive word that corresponds to a set of either $3$-sided or $4$-sided dice and $\tau=ABCCBA$. Then $(\tau)^k\sigma$ is an irreducible balanced non-transitive word.
\end{lemma}

\begin{proof}
Applying Lemma \ref{baby construction} $k$ times, we conclude that $(\tau)^k\sigma$ is balanced and non-transitive. It remains to prove that it is irreducible. Suppose that $(\tau)^k\sigma$ is reducible. Then we can write $(\tau)^k\sigma$ as $\pi\pi'$ where $\pi$ is an irreducible balanced non-transitive word and $\pi'$ is a balanced non-transitive word. Since $(\tau)^n$ is fair (so not non-transitive) and $(\tau)^n ABC$ is not balanced for any non-negative integer $n$, $\pi$ should be of the form $(\tau)^k \pi''$ ($\pi''$ is not empty). We write $(\tau)^k\sigma$ as $(\tau)^k\pi'' \pi'''$. Since $|\pi'''| < 12$, $\pi'''$ is irreducible. This implies that $\sigma=\pi''\pi'''$ corresponds to a balanced non-transitive set of $4$-sided dice and when it partially has a balanced non-transitive word $\pi'''$ of length less than $12$, which is impossible.  
\end{proof}

We are now ready to construct an irreducible balanced non-transitive set of $n$-sided dice.
\begin{theorem}\label{main construction2}
For any $n \geq 3$, there exists an irreducible balanced non-transitive set of $n$-sided dice.
\end{theorem}
\begin{proof}
We first consider the following two sets of dice:
$$
\begin{matrix}
A_3=9 & 5 & 1\\
B_3=8 & 4 & 3\\
C_3=7 & 6 & 2\\
\end{matrix}
\qquad\textrm{and}\qquad
\begin{matrix}
A_4= 10 & 7 & 5 & 4\\
B_4= 12 & 9 & 3 & 2\\
C_4= 11 & 8 & 6 & 1\\
\end{matrix}. 
$$
Both $D_3=(A_3, B_3, C_3)$ and $D_4=(A_4, B_4, C_4)$ are irreducible, balanced and non-transitive. As before we consider $\tau=ABCCBA$.
We now construct an irreducible balanced non-transitive set of $n$-sided dice for any $n \geq 3$. 
We already constructed such set of dice when $n=3$ and $n=4$.
When $n \geq 5$ is odd, we write $n = 3+ 2k, k \geq 1$. Then $(\tau)^kD_3$ is an irreducible balanced non-transitive set of $n$-sided dice due to Lemma \ref{main construction}. When $n \geq 5$ is even, we write $n = 4 + 2k, k \geq 1$. Then $(\tau)^kD_4$ is an irreducible balanced non-transitive set of $n$-sided dice due to Lemma \ref{main construction}.
\end{proof}

\section{Possible probability}\label{probability}
Let $(A,B,C)$ be a balanced non-transitive set of $n$-sided dice with $P(A>B)=P(B>C)=P(C>A)> \frac{1}{2}$. In this section our interest is on a possible probability $P(A>B)$. 
Let us state what is in our mind. 
\begin{conjecture}\label{conjecture prob}
If $(A, B, C)$ is a balanced non-transitive set of $n$-sided dice such that $P(A>B) > \frac{1}{2}$, then 
$$ \left( \frac{1}{2} < \,\right) \,\, P(A>B) < \frac{1}{2} + \frac{1}{9}.$$
\end{conjecture}

\subsection{Probabilities of our construction in Theorem \ref{main construction2}}
To support Conjecture \ref{conjecture prob}, we first calculate all possible probabilities of the set of dice which are constructed in Section \ref{irreducible}.

\begin{lemma}\label{recursion formula2}
Let $\sigma$ (resp. $\tau$) be a balanced non-transitive word of length $3m$ (resp. $3n$). Then 
\[
P_{\sigma\tau}(A>B) = \frac{1}{2} + \frac{\left(N_{\sigma}(A>B) - \frac{m^2}{2}\right) + \left(N_{\tau}(A>B) - \frac{n^2}{2}\right)}{(m+n)^2}.
\]
Similarly, we have 
\[
P_{\sigma\tau}(B>C) = \frac{1}{2} + \frac{\left(N_{\sigma}(B>C) - \frac{m^2}{2}\right) + \left(N_{\tau}(B>C) - \frac{n^2}{2}\right)}{(m+n)^2}
\]
and
\[
P_{\sigma\tau}(C>A) = \frac{1}{2} + \frac{\left(N_{\sigma}(C>A) - \frac{m^2}{2}\right) + \left(N_{\tau}(C>A) - \frac{n^2}{2}\right)}{(m+n)^2}.
\]
\end{lemma}

\begin{proof}
It is an easy consequence of the equation \eqref{recursion formula} since $N_{\sigma\tau}(A>B) = (m+n)^2 P_{\sigma\tau}(A>B)$.
\end{proof}
\begin{remark}
Note that $N_{\sigma}(A>B)$ represents the number of consequences that a die $A$ wins a die $B$. Therefore, $\left(N_{\sigma}(A>B) - \frac{m^2}{2}\right)$ in Lemma \ref{recursion formula2} represents how far it is away from being fair for a word $\sigma$. For example, if $N_{\sigma}(A>B) - \frac{m^2}{2}=0$, then $P_{\sigma}(A>B)= \frac{1}{2}$.
\end{remark}

\begin{lemma}\label{recursion formula3}
Let $\sigma$ and $\tau$ be as above. Assume that $\frac{1}{2} <P_{\sigma}(A>B)<a$ and $\frac{1}{2}<P_{\tau}(A>B)<b$. Then $\frac{1}{2} < P_{\sigma\tau}(A>B) < max \{ a, b \}$.
\end{lemma}
\begin{proof}
\[
(m+n)^2 P_{\sigma\tau}(A>B) = m^2 P_{\sigma}(A>B) + n^2 P_{\tau}(A>B)+mn <
\]
\[ max \{a, b\} (m^2+n^2 + \frac{1}{ max\{a, b\}}mn) <  max \{a, b\} (m+n)^2
\]
since $\frac{1}{2} < max \{a, b\}$.
\end{proof}

We are now ready to calculate the probability $P(A>B)$. Let $\sigma$ be a word that is constructed in Theorem \ref{main construction2}. Then $\sigma$ is either $(\tau)^k\tau_1$ or $(\tau)^k\tau_2$ for some non-negative integer $k$, where
\[
D(\tau_1)= 
\left\{ \begin{matrix}
9 & 5 & 1\\
8 & 4 & 3\\
7 & 6 & 2\\
\end{matrix} \right., \ \ 
D(\tau_2)= 
\left\{ \begin{matrix}
10 & 7 & 5 & 4\\
12 & 9 & 3 & 2\\
11 & 8 & 6 & 1
\end{matrix} \right., \ \
D(\tau)= 
\left\{ \begin{matrix}
6 & 1\\
5 & 2\\
4 & 3
\end{matrix} \right.
\]
Note that $(\tau)^k$ is a fair word (i.e., $P_{\tau^k}(A>B)=\frac{1}{2}$).
For $i=1, 2$, Lemma \ref{recursion formula} implies that 
\[
(2k+n_i)^2P_{\sigma}(A>B) =  (2k)^2P_{\tau^k}(A>B) + (n_i)^2 P_{\tau_i}(A>B) + 2kn_i= 
\]
\[
\frac{(2k)^2}{2} + \big[ \frac{n_i^2+2}{2}\big] + 2kn_i = \big[ \frac{(n_i+2k)^2+2}{2}\big]  
\]
which is closest integer greater than $\frac{(n_i+2k)^2}{2}$. Here, $3n_i$ is the length of $\tau_i$ and, therefore, $n_i = i+2$.

Therefore, Conjecture \ref{conjecture prob} is true in this case.
\begin{remark}
Similarly, we can also calculate possible probabilities of the set of dice that is constructed in \cite{SS17}.
Let $\sigma$ be a word that is constructed in Theorem 2.1 of \cite{SS17}. Then $\sigma$ is a product of several $\tau_1, \tau_2$, and $\tau_3$, where
\[
D(\tau_1)= 
\left\{ \begin{matrix}
9 & 5 & 1\\
8 & 4 & 3\\
7 & 6 & 2\\
\end{matrix} \right., \ \ 
D(\tau_2)= 
\left\{ \begin{matrix}
12 & 10 & 3 & 1\\
9 & 8 & 7 & 2\\
11 & 6 & 5 & 4
\end{matrix} \right., \ \ 
D(\tau_3)= 
\left\{ \begin{matrix}
15 & 11 & 7 & 4 & 3\\
14 & 10 & 9 & 5 & 2\\
13 & 12 & 8 & 6 & 1 
\end{matrix} \right.
\]
$P_{\tau_i}(A>B) = \frac{[ \frac{n_i^2 +2}{2} ]}{n_i^2}$ for $i=1, 2, 3$, where $n_i = (i+2)^2$. Therefore, Lemma \ref{recursion formula3} implies that $
P_{\sigma}(A>B) \leq \frac{5}{9} = \frac{1}{2} + \frac{1}{18}$. Therefore, Conjecture \ref{conjecture prob} is true in this case.
\end{remark}

\subsection{Best bound for probability}

We first show that $\frac12$ is the greatest lower bound of $P(A>B)$, which means that the inequality $P(A>B)>\frac12$ is optimal. 
\begin{theorem}
There exists a balanced non-transitive set of dice such that $P(A>B)-\frac12>0$ becomes arbitrary small.
\end{theorem}
\begin{proof}
Let $n=2m+1\,\,\,(m\in\mathbb{N})$. Then we construct a word $\sigma$ (or a corresponding dice) as follows. First consider the word $(ABCCBA)^m (BAC)$ which is not balanced, since the number of the events for $C$ to beat $A$ or $B$ is one more than the ones of the events for $A$ or $B$ to do $C$ (therefore $C$ always beats both $A$ and $B$ at least in probability). To change this word to a balanced one, let us replace $BC$ by $CB$ on the first factor $ABCCBA$. In all, the word $\sigma$ obtained by  
\begin{equation*}
 \sigma:=(ACBCBA)(ABCCBA)^{m-1}(BAC)
\end{equation*}
is balanced. A direct computation of $P(A>B)$ on $\sigma$, or more precisely, 
\begin{equation*}
    N_{\sigma(D)}(A>B)=0+2+\{(2+4)+\cdots+(2(m-1)+2m)\}+(2m+1)=2m^2+2m+1,
\end{equation*} 
shows that 
 \begin{equation*}
     P(A>B)=\frac{2m^2+2m+1}{(2m+1)^2}=\frac12+\frac{0.5}{n^2},
 \end{equation*}
which goes to $\frac12$ as $n\to\infty$. This indicates that  $\frac12$ is the greatest lower bound of $P(A>B)$ and the inequality $P(A>B)>\frac12$ could not get better.
\end{proof}
\medskip
 
For the upper bound of $P(A>B)$, recall that Conjecture \ref{conjecture prob} states that 
\begin{equation*}
  P(A>B) < \frac{1}{2} + \frac{1}{9}.
\end{equation*}
That is, the probability $P(A>B)$ cannot exceed (nor be equal to) $\frac{11}{18}$. 
To support this inequality, we provide three new constructions of sets of dice of which probability is close to $\frac12 + \frac19$.
For simplicity, assume that $n=6p$ (later we discuss when $n=6p+2$ or $n=6p+4$ which is similar to the presenting case).
Let us start with the most unmixed fair $n$-sided dice $\sigma_f$,  
\begin{equation}\label{unmixed fair}
   \sigma_f:= A\cdots A \, B\cdots B \, C\cdots C \, C\cdots C \, B\cdots B \, A\cdots A,
\end{equation}
where ``$\cdots$" means that all the letters are the same as the boundary letters and the numbers of letters in ``$\cdots$" are the same, i.e., $|A\cdots A|=|B\cdots B|=|C\cdots C|=3p$. 

We apply the following algorithm to construct most probability $P(A>B)$:
\noindent
\begin{algorithm} \quad
\begin{itemize}
    \item[(Step1)]  If there is no triple $AB, BC, CA$ in a word, then replace $xy$ and $yx$ in a word for $x, y \in \{A, B, C\}$ so that the resulting word contains $AB, BC, CA$.
    \item[(Step2)] If you find $AB, BC, CA$, then replace $AB, BC, CA$ by $BA, CB, AC$ respectively, but all together. 
   \end{itemize}
\end{algorithm}

\noindent
\begin{remark} \quad
\begin{itemize}
    \item Replacement in (Step1) does not change the probability due to Lemma \ref{symm exchange}.
    \item Replacement in (Step2) increases the number of events for $A$ (resp. $B$ or $C$) to beat $B$ (resp. $C$ or $A$) by 1. In particular, doing this replacement keeps the words balanced and non-transitive. Therefore we may expect that keeping replacing in this way would lead most probability of $P(A>B)$. 
\end{itemize}
\end{remark}


Since there is no $CA$ in $\sigma_f$, we first apply (Step1) (exchange of $BC$ and $CB$) several times to $\sigma_f$ in order to obtain a word $\sigma_1$ which is still fair but having sequels $CA$ by
\begin{eqnarray}\label{sigma1}
  \sigma_1:= &&  A\cdots A \, A\cdots A \, A\cdots A \, B\cdots B \, B\cdots B \, C\cdots C \\ \nonumber
  && \,\, C\cdots C \, C\cdots C \, {\bf {\color{red} B\cdots B}} \, C\cdots C \, C\cdots C \, B\cdots B \\  \nonumber 
  && \,\,\, B\cdots B \, B\cdots B \, {\bf {\color{red} C\cdots C}} \, A\cdots A \, A\cdots A \, A\cdots A \nonumber
\end{eqnarray}
(i.e., the last one third of the first $B\cdots B$ in \eqref{unmixed fair} moves to the middle and the last one third of the second $C\cdots C$ in \eqref{unmixed fair} does to the front of the second $A\cdots A$ (the moved ones are in red and bold above). Here, $|A \cdots A|=|B \cdots B|=|C \cdots C|=p$ in \eqref{sigma1}. Therefore the numbers of replacing $BC$ and $CB$ by $CB$ and $BC$ respectively are the same as $3p^2$.) 

We now apply (Step2) $3p^2$ times to obtain a balanced non-transitive word $\sigma_2$ by
\begin{eqnarray*}
  \sigma_2:= && {\bf {\color{red} B\cdots B}} \, A\cdots A \, A\cdots A \, A\cdots A \,  C\cdots C \, C\cdots C \\
  && \,\, C\cdots C \, {\bf {\color{red} B\cdots B}} \, B\cdots B \, C\cdots C \,  C\cdots C  \, B\cdots B \\
  && \,\,\, B\cdots B \, B\cdots B  \, A\cdots A \, A\cdots A \, A\cdots A \, {\bf {\color{red} C\cdots C}} 
\end{eqnarray*}
(i.e., in \eqref{sigma1} the first $B\cdots B$ moves in front, the second $B\cdots B$ does to almost middle and the last $C\cdots C$ does to the end. The moved ones are in red and bold as before.) Here the probability $P(A>B)$ of $\sigma_2$ increases to  $1/2+3p^2/(6p)^2=7/12$. 

Since there is no $AB$ and $CA$ in $\sigma_2$, we apply (Step1) $6p$ times as follows;
we move $B$ and $C$ in $\sigma_2$ to the ones in $\sigma_3$ (all of which are in bold below) in order to produce the products $AB$ and $CA$.
\begin{eqnarray*}
  \sigma_2:= && B\cdots B \, A\cdots A \, A\cdots A \, A\cdots A \,  C\cdots C \, C\cdots C \\
  && \,\, C\cdots C \, {\bf {\color{red} B}}\cdots B \, B\cdots B \, C\cdots C \,  C\cdots {\bf {\color{red} C}}  \, B\cdots B \\
  && \,\,\, B\cdots B \, B\cdots B  \, A\cdots A \, A\cdots A \, A\cdots A \, C\cdots C 
\end{eqnarray*}
\begin{eqnarray*}
  \sigma_3:= && B\cdots B \, A\cdots A \, A\cdots A \, A\cdots A \, {\bf {\color{red} B}}\, C\cdots C \, C\cdots C \\
  && \,\, C\cdots C \, {\dot B}\cdots {\color{blue}{{\ddot B} \, {\ddot B} \cdots {\ddot B} \, {\ddot C}\cdots {\ddot C} \,  {\ddot C}}}\cdots {\dot C}  \, B\cdots B \\
  && \,\,\, B\cdots B \, B\cdots B \, {\bf {\color{red} C}} \, A\cdots A \, A\cdots A \, A\cdots A \, C\cdots C\\
  && (\textrm{${\dot B}$ and ${\dot C}$ have been moved to ${\bf {\color{red} B}}$ and ${\bf {\color{red} C}}$, respectively.})
\end{eqnarray*}
This means that we need to exchange $BC$ to $CB$ in the middle part (which is blue and has two dots above the character on $\sigma_3$) $6p$ times ($3p$ to move $B$ and another $3p$ to do $C$). Then to increase $P(A>B)$, one more exchange from $BC$ to $CB$ in the middle part is necessary. 
 
In addition to this observation, we would like to measure how many times we can apply (Step2) to $\sigma_3$ in terms of $p$. (This is because the denominator of $P(A>B)$ is $(6p)^2$, so we would better express in $p$ in order to examine asymptotic behavior of $P(A>B)$.)
 Put by $m$ the number of further replacement such that $m$-many $B$'s move between $B\cdots B$ and $A\cdots A$ in front. For this we need extra $3pm$-many exchanges of $BC$ and $CB$ in the middle part. Since the remaining of each $B$'s or $C$'s in the middle is $2p-m$, the possible exchange of $BC$ to $CB$ is $(2p-m)^2$, which should be at least $9pm(=3pm+3pm+3pm)$ (the last $3pm$ is due to move $m$-many $B$'s in the further replacement). As a summary, we have that  
\begin{eqnarray*}
 && (2p-m)(2p-m)\ge 9pm \quad (m\leq 2p) \\
 && \quad \Longrightarrow \quad  m\leq \frac{13-\sqrt{154}}{2} \, p.
\end{eqnarray*}
Therefore the largest probability on $P(A>B)$ in the argument above is 
\begin{equation*}
    \frac12+\frac{3p^2+\frac{13-\sqrt{154}}{2} \, p \cdot 3p}{(6p)^2}=\frac12+\frac{15-\sqrt{154}}{24} \approx \frac12+\frac{1}{9.25} < \frac12+\frac19.
\end{equation*}
Note that the maximum $m$ (in $\mathbb N$) means that, since all $BC$ are exhausted in the last word, the probability $P(A>B)$ is not able to be larger in this construction.
\medskip

The cases when $n=6p+2$ and $n=6p+4$ can be investigated by a similar procedure. By applying (Step1) (or exchanging $BC$ by $CB$) and (Step2) (or replacing $AB$, $BC$ and $CB$ by $BA$, $CB$ and $BC$), the probability $P(A>B)$ increases by 
\begin{equation*}
    \frac{p(3p+1)}{[2(3p+1)]^2} \quad (\textrm{when } n=6p+2) \,\,\textrm{ or }\,\, \frac{p(3p+2)}{[2(3p+2)]^2} \quad (\textrm{when } n=6p+4)
\end{equation*}
(which is similar to the increase $p(3p)/(6p)^2$ when $n=6p$). 
In these cases, after doing this procedure, new words are expressed by
\begin{eqnarray*}
 n=6p+2: \,\,\sigma_2= && B\cdots B A\cdots A A\cdots A A\cdots A A {\bf {\color{red}B}} C\cdots C C\cdots C  \\
 && \,\, C\cdots C C B\cdots B B\cdots B C\cdots C C\cdots C C \cdots B\\
 && \,\,\,\, B\cdots B B\cdots B B A\cdots A A\cdots A A\cdots A A C\cdots C \\
 n=6p+4: \,\,\sigma'_2= && B\cdots B A\cdots A A\cdots A A\cdots A AA {\bf {\color{red}BB}} C\cdots C C\cdots C  \\
 && \,\, C\cdots C CC B\cdots B B\cdots B C\cdots C C\cdots C CC \cdots B\\
 && \,\,\,\, B\cdots B B\cdots B BB A\cdots A A\cdots A A\cdots A AA C\cdots C
\end{eqnarray*}
(Here each $\cdots$ has exactly $p-2$ letters. For example $|A\cdots A|=p-2+2=p$.)

Next, to make $P(A>B)$ most, we do (Step2) further as in the previous case. Shortly speaking, when $n=6p+2$, we have that 
\begin{eqnarray*}
 && (2p-m)(2p+1-m)-(3p+1) \ge 3m(3p+1) \quad (m\leq 2p) \\
 && \quad \Longleftrightarrow \quad m^2-(13p+4)m+(4p^2-p-1) \ge 0\\
 && \quad \Longrightarrow \quad  m\leq \frac{13p+4-\sqrt{153p^2+108p+20}}{2}.
\end{eqnarray*}
Hence the increased probability on $P(A>B)$ in the argument (i.e., except $1/2$) is 
\begin{equation*} 
    \frac{\left(p+\frac{13p+4-\sqrt{153p^2+108p+20}}{2}\right)(3p+1)}{(6p+2)^2} \xrightarrow{p\to\infty}\frac{1}{12}+\frac{13-\sqrt{153}}{24} <\frac{1}{9.12} <\frac19.
\end{equation*}
Note that, since the numerator $13p+4-\sqrt{153p^2+108p+20}$ is increasing in $p$, the last inequality holds for all $p$.  

Similarly when $n=6p+4$, we have that 
\begin{eqnarray*}
 && (2p-m)(2p+2-m)-2(3p+2) \ge 3m(3p+2) \quad (m\leq 2p) \\
 && \quad \Longleftrightarrow \quad m^2-(13p+8)m+(4p^2-2p-4) \ge 0\\
 && \quad \Longrightarrow \quad  m\leq \frac{13p+8-\sqrt{153p^2+216p+80}}{2}.
\end{eqnarray*}
Hence the increased probability on $P(A>B)$ for $n=6p+4$ (again except $1/2$) is 
\begin{equation*}
    \frac{\left(p+\frac{13p+8-\sqrt{153p^2+216p+80}}{2}\right)(3p+2)}{(6p+4)^2} \xrightarrow{p\to\infty}\frac{1}{12}+\frac{13-\sqrt{153}}{24} < \frac{1}{9.12} < \frac19.
\end{equation*}
Note again that the numerator $13p+8-\sqrt{153p^2+216p+80}$ increases in $p$.
\medskip

Let us summarize the argument above. Start with the most unmixed fair word $\sigma_f$ (which is \eqref{unmixed fair}). By replacing $AB$, $BC$, $CA$ by $BA$, $CB$, $AC$, i.e., (Step2) as much as possible, the probability $P(A>B) \,(=P(B>C)=P(C>A))$ increases most. Then the computation above tells us that 
$$
 P(A>B) < \frac{1}{2} + \frac{1}{9},
$$
which upholds Conjecture \ref{conjecture prob}.

\noindent
\begin{remark} \quad
\begin{itemize}
    \item Let us tell why we think the last constructed word in this argument provides the most probability on $P(A>B)$. First, observe that we have considered the only case when $n$ was even (when $n=6p$, $6p+2$ or $6p+4$). We, however, believe that this would be enough due to the following observation. A direct computation shows that there is only one possible probability $P(A>B)$ when $n=3$, or equivalently, any three balanced, non-transitive dice with 3 sides can be expressed by a balanced, non-transitive word which is similar to $\sigma_3=CBABACACB$. In this case $P_{\sigma_3}(A>B)=5/9$. Then add $CBA$ on the end of $\sigma_3$ and then exchange $CA$ (in $\sigma_3$) to $AC$. Then the constructed word $\sigma_4$ is $\sigma_4=CBABA{\color{red}AC}CB{\color{blue}CBA}$, which becomes balanced and non-transitive. For this, $P_{\sigma_4}(A>B)=9/16$, which is greater than $P_{\sigma_3}(A>B)=5/9$. A similar computation can be done for more general $n$. Therefore, with this trick (however we do not know if we can do this trick at all times), the probability of $P(A>B)$ would get greater for even $n$'s than odd ones.
    \item Second, proving that the above probability is the most probability is related to construction of all balanced non-transitive sets of dice and it also seems to be related to property of fair sets of dice, i.e., Conjecture \ref{fair dice} and we leave this for future work.
\end{itemize}
\end{remark}

\begin{remark}
In \cite{S18}, Schaefer further generalize the results in \cite{SS17} to sets of $n$-sided $m$ dice for any positive integer $m$. It will be very interesting to generalize our results to sets of $m$ dice for $m \geq 4$ and we also leave this for future work.
\end{remark}

\medskip



\subsection*{Acknowledgement}
The first-named author was supported by Basic Science Research Program through the National Research Foundation of Korea(NRF) funded by the Korea government  $($NRF-2019R1F1A1061300$)$. The second-named author was partially supported by Chonnam National University (Grant number: 2018-0978), and by the National Research Foundation of Korea (NRF) grant funded by the Korea government (MSIP) (No.2017R1C1B2010081).

\end{document}